\documentclass[preprint,12pt]{elsarticle}

\usepackage{amsfonts, amsmath, amssymb, amsthm, mathtools, extarrows, hyperref}
\usepackage[small,nohug,heads=LaTeX]{diagrams}
\hypersetup{
    colorlinks=true,
    linkcolor=blue,
    filecolor=magenta,      
    urlcolor=cyan,
    pdftitle={Overleaf Example},
    pdfpagemode=FullScreen,
    }
\diagramstyle[labelstyle=\scriptstyle]
\usepackage[section]{placeins}
\usepackage{upquote, nameref} 

\newtheorem{theorem}{Theorem}[section]
\newtheorem{lemma}[theorem]{Lemma}
\newtheorem{proposition}[theorem]{Proposition}
\newtheorem{corollary}[theorem]{Corollary}
\theoremstyle{definition} 
 
\newtheorem{definition}[theorem]{Definition}
\newtheorem{example}[theorem]{Example}
\newtheorem{remark}[theorem]{Remark}

\newcommand{\ra}{\longrightarrow}

\newcommand{\mfm}{\mathfrak m}
\newcommand{\mfc}{\mathfrak c}

\newcommand{\mce}{\mathcal{E}}
\newcommand{\mcf}{\mathcal{F}}

\newcommand{\mcl}{\mathcal{L}}

\newcommand{\mcy}{\mathcal{Y}}

\newcommand{\mbf}{\mathbb{F}}
\newcommand{\mbg}{\mathbb G}
\newcommand{\mbk}{\mathbb K}

\newcommand{\mbp}{\mathbb{P}}

\newcommand{\ints}{\mathbb{Z}}

\newcommand{\rats}{\mathbb{Q}}

\newcommand{\la}{\lambda}

\newcommand{\veps}{\varepsilon}
\newcommand{\vp}{\varphi}

\newcommand{\stl}{\left\{}
\newcommand{\str}{\right\}}

\newcommand{\ds}{\oplus}
\newcommand{\bds}{\bigoplus}
\newcommand{\im}{\text{im}}

\newcommand{\hm}{\text{Hom}}
\newcommand{\ckr}{\text{coker}}

\newcommand{\bu}{\bullet}

\newcommand{\floor}{\lfloor \frac{c}{2}\rfloor}

\newcommand{\soc}{\text{Soc}}
\newcommand{\E}{\mathcal{E}}
\newcommand{\Op}{\mathcal{O}}
\newcommand{\ol}{\overline}
\newcommand{\F}{\mathcal{F}}
\newcommand{\e}{\mathcal{E}}

\newcommand{\ml}{\mathcal L}

\journal{journal}

\begin{document}

\begin{frontmatter}



\title{Symmetry, Unimodality, and Lefschetz Properties for Graded Modules}


\affiliation{organization={Two Six Technologies}
            }
\author{Zachary Flores}
\ead{zachary.flores@twosixtech.com}

\begin{abstract}
We investigate the Weak Lefschetz Properties for modules whose minimal free resolutions are given by generalized Kosuzl complexes in dimension three through a careful study of their Betti numbers and the symmetry and unimodality of their Hilbert functions.  We also study the non-Lefschetz locus for finite length modules in arbitrary dimension, and are able to generalize several previous results on the non-Lefschetz locus in this setting.  Along the way, we find several connections with a Gorenstein analogue for finite length modules and Artin level modules that are both interesting and useful throughout this paper.  
\end{abstract}

\end{frontmatter}



\section{Introduction}\label{sec: intro}
We let $\mbk$ be an algebraically closed field, and $S$ the polynomial ring $\mbk[x_1,\ldots, x_r]$ with standard grading and homogeneous maximal ideal $\mfm = (x_1,\ldots, x_r)$.  All $S$-modules considered are finitely generated, in particular, if a finitely generated $S$-module is Artinian, then it has finite length.  We begin with the following definition.  

\begin{definition}\label{def: wlp} 
If $N$ is a graded Artinian $S$-module, we say that \textbf{$N$ has the Weak Lefschetz Property} (\textbf{WLP)} if there is a general linear form $\ell \in S_1$ such that the $\mbk$-linear map $\times \ell:  N_j\ra N_{j+1}$ has maximal rank for all $j$.  
\end{definition} 

The WLP has been studied extensively in the case that $N$ is cyclic, see \cite{MN} for an excellent overview.  Despite the fact that it is not difficult to define the WLP for graded Artinian modules over $S$, before a preprint to this article existed, the WLP had rarely been studied for Artinian modules over $S$.  Work in this direction included \cite{FD}, where the authors study the WLP for an Artinian  graded module over $S$ for $r = 2$, and give an algorithm to test whether or not a graded Artinian module with fixed Hilbert function has the WLP.  
 
The motivation of our work is to generalize the main result of \cite{CI} that shows that if $\mbk$ has characteristic zero, $r = 3$, and $I$ is a codimension $3$ complete intersection, then $N = S/I$ has the WLP.  We aim to generalize the result in the following manner:  if $R = \mbk[x,y, z]$ and $M$ is the finite length cokernel of an $R$-linear map $\vp:  \bds_{j=1}^{n+2} R(-b_j) \ra \bds_{i=1}^n R(-a_i)$, then $M$ has the WLP.  If $n = 1$, then $M$ quotient of $R$ by a complete intersection, and we recover (\cite{CI}, Theorem 2.3).  We were successful in this direction (see Theorem \ref{thm: mainWLPthm}) for a large class of $\vp$, but there are still restrictions on the $a_i$ and $b_j$.  These restrictions were removed in (\cite{ffp}, Theorem 3.7), but we see our results having a distinct algebraic flavor.  Moreover, we are also able to calculate all Betti numbers of $M$ in Section \ref{sec: minFreeRes}, as well provide proofs of other algebraic results that \cite{ffp} does not touch that we discuss below.  

Our results in this direction were inspired by the proof in \cite{CI} that the quotient of a complete intersection has the WLP.  There subtle but strong use is made of the fact that the Hilbert function of a complete intersection is both symmetric and unimodal, and it is unclear when our Artinian module in question may have symmetric or unimodal Hilbert function.  

The symmetry of the Hilbert function of a complete intersection comes from the fact that a complete intersection is Gorenstein, and Gorenstein algebras are well-known to have symmetric Hilbert functions.  There is not a widely-known analogue for the Gorenstein condition for modules of finite length, however, there is a proposed analogue defined in \cite{MK} (see Definition \ref{def: symgor}) that suits our needs perfectly.  Using Definition \ref{def: symgor} and a main result of \cite{MK} (Theorem \ref{thm: MKthm}) we are able to determine when $M$ has symmetric Hilbert function in Proposition \ref{prop: symhilb}.  Moreover, using this, we are able to determine when $M$ has unimodal Hilbert function in Proposition \ref{prop: unimodal}.  While the use of such results was to determine when $M$ has the WLP, they hold independently of $M$ having the WLP, unlike (\cite{ffp}, Corollary 3.9).  
 
Inspired by \cite{nll}, we define and discuss the \textbf{non-Lefschetz locus} for an Artinian graded $S$-module $N$.  Here, given an Artinian $S$-module $N = \bds_{j\in\ints} N_j$, the $S$-module structure of $N$ is determined by a sequence of $\mbk$-linear maps $\phi_j: S_1\ra \hm_\mbk(N_j, N_{j+1})$.  In particular, given a linear form $\ell = a_1x_1 + \cdots + a_rx_r$, $\phi_j(\ell)$ is a matrix of linear forms in $a_1,\ldots, a_r$.  Regarding $a_1,\ldots, a_r$ as variables, we look at the scheme defined by the vanishing of the maximal minors of this matrix, and this is our object of study.  Our work in this direction discusses some issues that are raised when attempting to generalize results of \cite{nll}, but make use of some connections with results on \textbf{Artin level modules} from \cite{alm}, that we also find are of independent interest.  We note that much of our work on the non-Lefschetz locus was generalized by \cite{EM}, but \cite{EM} cited an early preprint of this article heavily, and some of their results rely strongly on the work in this paper.   

The current paper is organized as follows.  In Section \ref{sec: minFreeRes}, we compute the minimal free resolution of a graded $R$-module, for $R = \mbk[x,y,z]$, of the cokernel $M$ of a map $\vp:  \bds_{j=1}^{n+2} R(-b_j)$ to $\bds_{i=1}^n R(-a_i)$ when it has finite length.  This is essential for Section \ref{sec: UniAndSym}, where we discuss symmetry and unimodality properties of $M$, most notably using an analogue of the Gorenstien condition for Artinian modules defined in \cite{MK}.  In Section \ref{sec: LefProp}, we discuss when the $R$-module $M$ has the WLP, recover (\cite{CI}, Theorem 2.3), and give an example a family of non-cyclic $R$-modules that have the WLP using circulant matrices.  In Section \ref{sec: nonLefLocus}, we discuss the non-Lefschtez locus for a graded $S$-module $N$ and give some generalizations from work in \cite{nll}.  Lastly, we discuss what conditions we can place on $N$ so that is the non-Lefschetz locus is given by at most two degrees, and, in some cases, a single degree.  

\section{The Minimal Free Resolution of $M$}\label{sec: minFreeRes}

Our setup for this section is as follows:  $R$ is the polynomial ring $\mbk[x, y, z]$, where $\mbk$ is algebraically closed (we will restrict the characteristic when necessary), $n > 0$ is a positive integer, 

\begin{equation}\label{eq: defM}
\vp:  \bds_{j=1}^{n+2} R(-b_j)\ra \bds_{i=1}^n R(-a_i)
\end{equation}

is a degree zero graded homomorphism with cokernel $M$ such that $b_1\leq b_2\leq\cdots\leq b_{n+2}$ and $a_1\leq a_2\leq \cdots \leq  a_n$, the map $\vp = (\vp_{i,j})$ is such that either $\vp_{i,j} = 0$ or $\vp_{ij}\in R_{e_{ij}}$ with $e_{ij} > 0$ (hence $\deg(\vp_{i} = b_j- a_i)$, and if $I$ denotes the ideal generated by the $n\times n$ minors of $\vp$, we assume that that $I$ has codimension $3$, so that $M$ is Artinian, hence of finite length.  

Since $I$ has codimension $3$, by (\cite{E1}, Theorem A.210), the Buchsbaum-Rim complex provides the minimal free resolution of $M$.  That is, there is an exact sequence $\mbf_\bu$:  

\begin{equation}\label{eq: minFreeRes}
0\ra \bds_{i=1}^n R(-d_i) \stackrel{\delta}{\ra}\bds_{j=1}^{n+2} R(-c_j)\stackrel{\varepsilon}{\ra} \bds_{j=1}^{n+2} R(-b_j)\stackrel{\vp}{\ra} \bds_{i=1}^n R(-a_i),
\end{equation}

where the entries of all maps live in $\mfm$.  In this section, we want to calculate all graded Betti numbers of $M$, which is the cokernel of $\vp$; that is, we determine the values of the $c_j$ and $d_i$.  To do so, we first need information about the maps $\veps$ and $\delta$.  Before we proceed, we note the following lemma that will be used frequently in the sequel.    

\begin{lemma}\label{lem: codim3}
If $\vp:  \bds_{j=1}^{n+2} R(-b_j) \ra \bds_{i=1}^n R(-a_i)$ is as in \eqref{eq: defM}, then $b_i > a_i$ for $i = 1,\ldots, n$.   
\end{lemma} 

\begin{proof} 
Suppose not.  Then there is an $i$ such that $b_i\leq a_i$.  We recall that $b_1\leq \cdots \leq b_{n+2}$ and $a_1\leq \cdots\leq a_n$, so this gives 

$$b_1\leq\cdots\leq b_{i-1}\leq b_i\leq a_i\leq a_{i+1}\leq\cdots\leq a_n.$$

In particular, $\vp$ contains a zero submatrix of size $(n-i+1)\times i$.  Let $t(\vp)$ denote the half-perimeter of this zero submatrix, so that $t(\vp) = n+1$.  Then (\cite{sparse}, Th\'{e}or\`{e}me 1.6.2) says that the codimension of $I$ is at most $n +3-t(\vp)$.  In particular, $I$ has codimension at most $2$, contrary to our assumption.   
\end{proof} 

\subsection{The map $\veps$}
For ease of notation, set $\mbf_1 = \bds_{j=1}^{n+2} R(-b_j)$ and $\mbf_2= \bds_{j=1}^{n+2} R(-c_j)$.  Let $f_{1,1},\ldots, f_{1, n+2}$ be a basis for $\mbf_1$ and $f_{2,1},\ldots, f_{2,n+2}$ be a basis for $\mbf_2$.  Then by (\cite{E1}, Section A2.6.1) $\veps$ is the map such that 

$$\varepsilon(f_{2,j}) = \sum_{K_{r,j}\subset H_j}\text{sgn}(K_{r,j} \subset H_j) \det(\vp_{K_{r,j}}) f_{1,r}$$

Where for $j=1,\ldots, n+2$, $H_j = \stl 1,\ldots, n+2\str\backslash \stl j\str$; for $r\in H_j$, $K_{r,j} = H_j\backslash \stl r\str$; $\vp_{K_{r,j}}$ is the the $n\times n$ minor of $\vp$ indexed by the elements of $K_{r,j}$; and $\text{sgn}(K_{r,j}\subset H_j)$ is the sign of the permutation of $H_j$ that puts the elements of $K_{r,j}$ into the first $n$ positions of $H_j$.  Thus the $j$th column of a matrix $\veps$ is given by 

$$\left[\begin{matrix}
\text{sgn}(K_{1j}\subset H_j)\det(\vp_{K_{1j}})\\
\vdots\\
\text{sgn}(K_{j-1, j}\subset H_j) \det(\vp_{K_{j-1, j}})\\
0\\
\text{sgn}(K_{j+1, j} \subset H_j) \det(\vp_{K_{j+1, j}})\\
\vdots\\
\text{sgn}(K_{n+2,j}\subset H_j)\det(\vp_{K_{n+2,j}})\\
\end{matrix}\right]$$

Noting the $0$ occurs in the $jth$ row.  When $1\leq r < j$, it is not hard to see that $\text{sgn}(K_{r,j}\subset H_j) = (-1)^{n-r+1}$.  Now for $ j < r\leq n+2$, it is also easy to see we have $\text{sgn}(K_{r,j}\subset I_j) = (-1)^{n-r+2} = (-1)^{n-r}$.  If $\Phi_{r,j} := \det(\vp_{K_{r,j}})$, then the $j$th column of the matrix of $\veps$ is given by 

\begin{equation}\label{eq: vepsMat}
\left[\begin{matrix}
(-1)^n\Phi_{1,j}\\
\vdots\\
(-1)^{n+2-j}\Phi_{j-1, j}\\
0\\
(-1)^{n +1 -j} \Phi_{j+1, j}\\
\vdots\\
\Phi_{n+2,j}\\
\end{matrix}\right]
\end{equation}

\subsection{The map $\delta$}
For ease of notation, set $\mbf_3 = \bds_{i=1}^n R(-d_i)$ and let $f_{3,1},\ldots, f_{3,n}$ be a basis for $\mbf_3$.  By (\cite{E1}, Section A.2.6.1) the map $\delta:  \mbf_3\ra \mbf_2$ is such that 

$$f_{3,i} \mapsto \sum_{j=1}^{n+2} (-1)^{j+1}\vp_{i,j}f_{2,j}.$$ 

In particular, the $i$th column of the matrix for $\delta$ is given by 

$$\left[\begin{matrix}
\vp_{i,1}\\
-\vp_{i,2}\\
\vdots\\
(-1)^{j+1}\vp_{i,j}\\
\vdots\\
(-1)^{n+2}\vp_{i,n+1}\\
(-1)^{n+3}\vp_{i, n+2}\\
\end{matrix}\right].$$

Thus a matrix for $\delta$ is given by 

$$\left[\begin{matrix}
\vp_{1,1} & \vp_{2,1} & \cdots & \vp_{n,1} \\
-\vp_{1,2} & -\vp_{2,2} & \cdots & -\vp_{n,2} \\ 
\vdots & \vdots & \cdots & \vdots \\
(-1)^{n+2}\vp_{1, n+1} & (-1)^{n+2}\vp_{2, n+1} & \cdots & (-1)^{n+2}\vp_{n, n+1}\\ 
(-1)^{n+3}\vp_{1, n+2} & (-1)^{n+3}\vp_{2, n+2} & \cdots & (-1)^{n+3}\vp_{n, n+2}\\ 
\end{matrix}\right]$$ 

\subsection{Computing the $c_j$ and $d_i$}
We first calculate the degrees of the $\Phi_{r,j}$.  This follows from the following general lemma, which is probably well-known, but we could not find an exact source.  

\begin{lemma}\label{lem: detForm}
Let $S = \mbk[x_1,\ldots, x_r]$ and  $\alpha:   \bds_{i=1}^t S(-v_i) \ra \bds_{i=1}^t S(-u_i)$ be a homogeneous $S$-linear map such that $v_i > u_i$.  If $\alpha$ has matrix 

$$\left[\begin{matrix}
\alpha_{11} & \alpha_{12} &\cdots &\alpha_{1t}\\
\alpha_{21} & \alpha_{22} &\cdots &\alpha_{2t}\\
\vdots & \vdots & \vdots & \vdots\\
\alpha_{t1} & \alpha_{t2} & \cdots & \alpha_{tt} 
\end{matrix}\right]$$

such that either $\alpha_{ij} = 0$ or $\alpha_{ij} \in S_{t_{ij}}$ with $t_{ij} > 0$, we denote the determinant of $\alpha$ by $\Phi$ and assume $\Phi$ is nonzero.  Then $\Phi$ is homogeneous of degree $\sum_{i=1}^t [v_i - u_i]$.  
\end{lemma} 

\begin{proof} 
Before we begin, notice that if $\alpha_{ij}$ is nonzero, then $\deg(\alpha_{ij}) = t_{ij} = v_j-u_i > 0$. 

We proceed by induction on $t$.  For $t = 1$, this is just the statement that a graded map $S(-v_1) \ra S(-u_1)$ is given by multplication of a homogeneous element of $S$ of degree $v_1 - u_1$.  This is easy to see.  Suppose that $t > 1$ and write 

$$\Phi = \alpha_{11}\Phi_1 - \alpha_{12}\Phi_2 \cdots + (-1)^{t+1}\alpha_{1t}\Phi_t$$

Where $\Phi_i$ is the  determinant of the $(t-1)\times (t-1)$ submatrix of $\alpha$ obtained by deleting the first row and the $i$th column.  By hypothesis, $\Phi$ is nonzero, so that there is an $h$ such that both $\alpha_{1h}$ and $\Phi_h$ are nonzero.  In this case, note that $\Phi_h$ is the determinant of a homogeneous linear map from $\bds_{j\neq  h}S(-v_j)$ to $\bds_{i\neq 1}S(-u_i)$.  The induction hypothesis gives that $\Phi_h$ is homogeneous of degree $\sum_{j\neq h} v_j - \sum_{i\neq 1} u_i$, hence $\alpha_{1h}\Phi_h$ is homogeneous of degree $\sum_{i=1}^t [v_i - u_i]$, as needed.  
\end{proof} 

Set $d = \sum_{j=1}^{n+2} b_j -\sum_{i=1}^n a_i$, so that Lemma \ref{lem: detForm} gives the following.  

\begin{corollary}\label{cor: corcj}
Let $\Phi_{r,j}$ be the maximal minor of $\vp$ corresponding to the set $K_{r,j}= H_j \backslash \stl r \str = \stl 1,\ldots, n+2\str \backslash \stl r, j\str$ (so that $\Phi_{r,j}$ is the minor of $\vp$ obtained by deleting columns $r$ and $j$ of $\vp$).  If $\Phi_{r,j}$ is nonzero, then the degree of $\Phi_{r,j}$ is $d - b_r - b_j$.  
\end{corollary} 

Suppose given $1\leq j\leq n+2$ that there is an $r\in H_j$ such that $\Phi_{r,j}\neq 0$.  Then we have $c_j = b_r + \deg(\Phi_{r,j}) = d - b_j$.  Thus we need to know if for all $j$, there is an $r\in H_j$ such that $\Phi_{rj}$ is nonzero.  We do this below.  

\begin{lemma}\label{lem: nonzerominor}
Given $1\leq j\leq n+2$ there is an $r\in H_j$ such that $\Phi_{r,j}$ is nonzero.  In particular, $c_j = d-b_j$. 
\end{lemma} 

\begin{proof} 
The sequence $\mbf_\bu$ in \eqref{eq: minFreeRes} is exact, so that if no $\Phi_{r,j}$ is nonzero, then the $j$th column of $\veps$ is zero.  This implies that $\textbf{e}_j := [0,\ldots, 1,\ldots, 0]^T \in \mbf_2$ is in $\ker(\veps)$, where the lone $1$ occurs in row $j$.  By the exactness of $\mbf_\bullet$, we can write $ \textbf{e}_j = \delta(\beta)$, where $\beta = [\beta_1,\ldots, \beta_n]^T\in  \mbf_3$.  This gives the equation,

$$\sum_{i=1}^n \vp_{ij}\beta_i = (-1)^{j+1}.$$ 

This gives a contradiction, as the sum on the left is either homogeneous of positive degree or zero.  
\end{proof} 

\begin{corollary}\label{cor: cordi}
$d_i = d-a_i$
\end{corollary} 

\begin{proof} 
Up to sign of entries, the $i$th column of the matrix for $\delta$ is the $i$th row of the matrix $\vp$, and since the $i$th row of the map $\vp$ has to be nonzero, there is a $j$ such that 

\begin{equation}\label{eq: dicj}
b_j - a_i  = \deg(\vp_{ij})  = d_i - c_j.
\end{equation}

By Lemma \ref{lem: nonzerominor}, $c_j = d - b_j$, so that using \eqref{eq: dicj}, $d_i =  d - a_i$.  

\end{proof}

\section{The Unimodality and Symmetry of the Hilbert Function of $M$}\label{sec: UniAndSym}

As previously mentioned, our motivation for wanting to study to the unimodality and symmetry of the $R$-module $M$ was to understand when $M$ has the WLP.  However, the question of whether or not a graded Artinian module $N$ over $S = \mbk[x_1,\ldots, x_r]$ has the WLP is more subtle if $N$ is not generated in a single degree.  For example, let $N$ be an Artinian $S$-module with Hilbert function $h_N$ such that $N_{j+1}$ contains a minimal generator of $N$ and $h_N(j) \geq h_N(j+1)$.  Then $\times\ell:  N_j\ra N_{j+1}$ cannot be surjective.  

Naturally, we would like to avoid such situations, and understand when $M$ has a \textit{strictly} unimodal Hilbert function over $R$; that is, where it is increasing or decreasing, it does so strictly.  In particular, we look for numerical conditions on the $a_i$ and $b_j$ that make it so that the Hilbert function of $M$ is strictly unimodal and symmetric. 

 The following lemma will be used frequently.  Its proof is essentially that of (\cite{soc}, Lemma 1.3), but we provide details.  

\begin{lemma}\label{lem: socle}
Let $S = \mbk[x_1,\ldots, x_r]$ and $N$ be a graded Artinian $S$-module with minimal free resolution $\mbg_\bu$.  If $\mbg_r = \bds_{j=1}^v S(-u_j)$ is the last module occurring in $\mbg_\bu$, then there is a homogeneous isomorphism 
$$\soc(N) \cong \bds_{j=1}^v \mbk(-(u_j-r))$$ 
\end{lemma}  

\begin{proof} 
We have $\text{Tor}^S_r(N, \mbk) = H_r(\mbg_\bu\otimes \mbk) = \bds_{j=1}^v \mbk(-u_j)$.  If $\mathbb C_\bu$ is the Koszul complex on $x_1,\ldots, x_r$, then we also have $\text{Tor}^S_r(N, \mbk) = H_r(\mathbb C_\bu\otimes N) = \soc(N)(-r)$.  
\end{proof} 

With Corollary \ref{cor: cordi} in hand, the following is immediate from Lemma \ref{lem: socle}.  

\begin{corollary}\label{cor: maxsocdeg} 
$M$ has maximal socle degree $d-a_1-3$.  
\end{corollary}  

We turn our discussion to graded duals of Artinian modules over $S = \mbk[x_1,\ldots, x_r]$.  

\begin{definition}\label{def: dual}
Let $N$ be a graded Artinian $S$-module.  Denote by $N^\vee$ the $S$-module $\hm_\mbk(N, \mbk)$.  Then $N^\vee$ is a graded $S$-module with $N^\vee_j = \hm_\mbk(N_{-j}, \mbk)$, and is called the \textbf{graded $\mbk$-dual} of $N$, or just \textbf{dual} of $N$ when the context is clear.  

The $S$-module action on $N^\vee$ is such that for $a\in S_i$ and $f\in N^\vee_j$, then $af\in N^\vee_{i+j}$ is the $\mbk$-linear map from $N_{-i-j}\ra \mbk$ with $(af)(\la) = f(a\la)$ for $\la\in N_{-i}$.
\end{definition} 

Following \cite{MK}, we now define an analogue of the Gorenstein condition for Artinian $S$-modules.

\begin{definition}\label{def: symgor}
A graded Artinian $S$-module $N$ is \textbf{Symmetrically Gorenstein} if there is an isomorphism $\tau:  N\ra \hm_\mbk(N, \mbk)(-s)$ such that $\tau = \hm_\mbk(\tau, \mbk)(-s)$. 
\end{definition} 

With the above definition in hand, consider the following.

\begin{lemma}\label{lem: hilbfunction}
Let $N$ be a non-negatively graded Artinian $S$-module, say $N = N_0 \ds\cdots\ds N_c$.  We suppose that $N_0$ and $N_c$ are nonzero.  Suppose there is a graded isomorphism $\tau:  N\overset{\cong}{\ra} N^\vee(-s)$ for some $s\in\ints$.  That is, $\tau(N_j)\subseteq N^\vee(-s)_{j+d}$ for some $d\in\ints$.  Then $N$ has symmetric Hilbert function.  
\end{lemma}

\begin{proof} 
We have $\tau(N_0)\subseteq N^\vee_{d-s}$, which gives $-c\leq d-s \leq 0$, as $N^\vee$ is concentrated in degrees $-c$ to $0$.  Also, $\tau(N_c) \subseteq N^\vee_{c+d-s}$ and $\tau(N_c)$ is nonzero, so we have $-c\leq c+d -s \leq 0$.  Thus $s - c = d$, which gives $\tau(N_j)\subseteq N^\vee(-s)_{j+s - c} = N^\vee_{j-c} = \hm_\mbk(\mbk, N_{c-j})$.  Hence we obtain an isomorphism of vector spaces over $\mbk$:  

$$\tau|_{N_j}:  N_j \ra \hm_\mbk(N_{c-j}, \mbk)$$ 

Thus for $ j = 0,1, \ldots, \lfloor \frac{c}{2}\rfloor$, we obtain $\dim_\mbk N_j = \dim_\mbk \hm_\mbk(N_{c-j}, \mbk) = \dim_\mbk N_{c-j}$.  That is, the Hilbert function of $N$ is symmetric.  
\end{proof} 

In particular, Lemma \ref{lem: hilbfunction} gives that Hilbert function of a non-negatively graded Symmetrically Gorenstein $S$-module in which the component in degree zero is nonzero is symmetric.     

Our goal throughout the rest of this section is to understand when our $R$-module in question, $M$, is Symmetrically Gorenstein.  Since we have a complete understanding of the free resolution of $M$ from Section \ref{sec: minFreeRes}, we us the following result that characterizes when a graded Artinian $S$-module is Symmetrically Gorenstein module in terms of its minimal free resolution.  

\begin{theorem} (\cite{MK}, Theorem 1.3)\label{thm: MKthm}
Suppose $\mbk$ has characteristic not two.  Let $S = \mbk[x_1,\ldots, x_r]$ and $N$ be a graded Artinian $S$-module with maximal socle degree $c$.  Set $d = c+r$ and $(\bullet)^{\vee d} = \hm_S(\bullet, S(-d))$.  Let $a\geq 3$ be an odd integer and $b = \frac{a-1}{2}$.  Then $N$ is Symmetrically Gorenstein if and only if its minimal graded free resolution has the following form 

$$0\ra (\mbg_0)^{\vee d} \stackrel{\psi_1^{\vee d}}{\ra}(\mbg_1)^{\vee d}\ra \cdots\ra (\mbg)_b^{\vee d} \stackrel{\psi_b^{\vee d}}{\ra} (\mbg)_b \ra  \cdots \ra \mbg_1\stackrel{\psi_1}{\ra} \mbg_0.$$
\end{theorem}

To this end, we utilize Theorem \ref{thm: MKthm} to show that under mild restrictions, $M$ is a Symmetrically Gorenstein $R$-module, so that by Lemma \ref{lem: hilbfunction}, $M$ will have a symmetric Hilbert function.  

We want to construct a minimal free resolution for $M$ with form as in Theorem \ref{thm: MKthm}, and to do that, note the following.  

\begin{remark}\label{rmk: antisym} 
For the minimal free resolution of $M$ in \eqref{eq: minFreeRes}, we write $\veps = [\Phi_1,\ldots, \Phi_{n+2}]$, with $\Phi_j$ the $j$th column of $\veps$.  

Consider the map $\veps':  \bds_{j=1}^{n+2} R(-c_j)\ra \bds_{j=1}^{n+2} R(-b_j)$ with 

$$\veps' = [-\Phi_1,\ldots, (-1)^j\Phi_j,\ldots, (-1)^{n+2}\Phi_{n+2}].$$ 

For $1\leq j < r$, using \eqref{eq: vepsMat}, we have 

$$\veps'_{rj} = (-1)^{n-r+j}\Phi_{rj}$$ 

$$\veps'_{jr} = (-1)^{n-j+1+r}\Phi_{jr}$$

Thus $\veps'_{jr} = -\veps'_{rj}$, so $\veps'$ is antisymmetric.    
\end{remark} 

We utilize $\veps'$ for the following.  

\begin{lemma}\label{lem: antires} 
With $\veps'$ from Remark \ref{rmk: antisym}, the sequence 

$$\mbf'_\bu:  0\ra \mbf_3 \stackrel{g'\delta}{\ra}\mbf_2\stackrel{\varepsilon'}{\ra} \mbf_1 \stackrel{\vp}{\ra} \mbf_0\ra M\ra 0$$ 

is exact.  Here, we recall $\mbf_2 = \bds_{j=1}^{n+2} R(-c_j)$, and $g': \mbf_2\ra \mbf_2$ is the map such that 

$$g'\left[\begin{matrix} \beta_1\\ \vdots\\ \beta_j\\ \vdots \\ \beta_{n+2}\end{matrix}\right] = \left[\begin{matrix} -\beta_1\\ \vdots\\(-1)^j\beta_j\\ \vdots \\ (-1)^{n+2}\beta_{n+2}\end{matrix}\right].  $$ 

In particular, there is an isomorphism of minimal free resolutions of $M$

$$\mbf_\bullet\cong \mbf'_\bu$$ 
\end{lemma} 

\begin{proof} 
From \eqref{eq: minFreeRes}, we know the sequence $\mbf_\bu$, 

$$0\ra \bds_{i=1}^n R(-d_i) \stackrel{\delta}{\ra}\bds_{j=1}^{n+2} R(-c_j)\stackrel{\varepsilon}{\ra} \bds_{j=1}^{n+2} R(-b_j)\stackrel{\vp}{\ra} \bds_{i=1}^n R(-a_i),$$ 

is exact and the cokernel of $\vp$ is $M$.  Clearly $g'\delta$ is injective, since $g'$ is an isomorphism.  Since $\veps'g' = \veps$, this gives $\im(\veps') = \im(\veps) = \ker(\vp)$.  Since

$$\veps'g'\delta = \veps\delta = 0,$$ 

this gives $\im(g'\delta)\subseteq \ker(\veps')$.  If $\veps'(\alpha') = 0$, then $\alpha' = g'(\alpha)$, for some $\alpha$ necessarily in $\ker(\veps)$ (as $g'$ is its own inverse), so that $\alpha = \delta(\beta)$, for some $\beta\in \bds_{i=1}^n R(-d_i)$.  That is, $\alpha' = g'\delta(\beta)$.  Thus $\mbf'_\bu$ is exact, which gives that $\mbf'_\bu$ is a graded minimal free resolution of $M$, whence the isomorphism of complexes.  
\end{proof} 

\begin{proposition}\label{prop: symhilb} 
The $R$-module $M$ is Symmetrically Gorenstien and its Hilbert function of $M$ is symmetric if $a_1 = 0$ and $\mbk$ has characteristic not two.  
\end{proposition} 

\begin{proof} 
By Corollary \ref{cor: maxsocdeg}, the maximal socle degree of $M$ is $d-3$.  As in the statement of Theorem \ref{thm: MKthm}, we let $(\bullet)^{\vee d}$ be the functor $\hm_R(\bullet, R(-d))$.  By Lemma \ref{lem: antires}, the complex $\mbf_\bullet'$,

$$0\ra \bds_{i=1}^n R(-d_i) \ra\bds_{j=1}^{n+2} R(-c_j)\stackrel{\veps'}{\ra} \bds_{j=1}^{n+2} R(-b_j)\ra \bds_{i=1}^n R(-a_i)$$

is the graded minimal free resolution of $M$.  By Corollary \ref{cor: corcj}, $c_j = d - b_j$ and by Corollary \ref{cor: cordi}, $d_i = d-a_i$.  This gives, 

\begin{eqnarray*}\label{eq: resCheck1 }
\left(\bds_{j=1}^{n+2} R(-b_j)\right)^{\vee d} &=& \bds_{j=1}^{n+2} \hm_R(R(-b_j), R(-d)) \\
&=& \bds_{j=1}^{n+2} R(b_j-d) \\
&=& \bds_{j=1}^{n+2} R(-c_j),\\
\end{eqnarray*}

and, 
\begin{eqnarray*}
\left(\bds_{i=1}^n R(-a_i)\right)^{\vee d} &=&\bds_{i=1}^n \hm_R(R(-a_i), R(-d))\\
&=& \bds_{i=1}^n R(a_i-d)\\ 
&=& \bds_{i=1}^n R(-d_i).  
\end{eqnarray*}

Thus the minimal graded free resolution of $M$ can be put into the form

$$0\ra \mbf_0^{\vee d} \ra \mbf_1^{\vee d} \overset{\veps'}{\ra} \mbf_1\ra \mbf_0\ra M. $$ 

The map $\veps'$ is antisymmetric by Remark \ref{rmk: antisym}, hence by Theorem \ref{thm: MKthm}, $M$ is Symmetrically Gorenstein.  By our assumption that $a_1 = 0$, $M$ is non-negatively graded and $M_0\neq 0$.  By Lemma \ref{lem: hilbfunction}, we obtain that the Hilbert function of $M$ is symmetric.  
\end{proof} 

Proposition \ref{prop: symhilb} answers the question of when the Hilbert function is symmetric.  This was a subtle but crucial point in showing that complete intersections in $R$ have the WLP in \cite{CI}.  However, as mentioned at the beginning of this section, a decreasing Hilbert function and having generators in degree greater than zero may cause $M$ to lack the WLP.  Our next goal is to show this is not the case, that is, Hilbert function of $M$ is unimodal, as needed.  

We begin with the following lemma.  

\begin{lemma}\label{lem: hilbFunForm}
Suppose $\mbk$ has characteristic not two, $d = \sum_{j=1}^{n+2} b_j -\sum_{i=1}^n a_i$, and $d' = \sum_{i=1}^n (b_i -a_i)$.  Then if 

\begin{enumerate}
    \item[(a)]$d$ is even and $ d' + b_{n+1} + 2 > b_{n+2}$; 
    \item[(b)] $d$ is odd and $d' + b_{n+1} + 1 > b_{n+2}$,
\end{enumerate}

for $t\leq \floor$, the Hilbert function $h_M(t)$ of $M$ is given by 

\begin{equation}\label{eq: hilbFun}
\sum_{i=1}^n {t+2-a_i\choose 2}-\sum_{j=1}^{n+2} {t+2-b_j\choose 2}
\end{equation} 
\end{lemma}

\begin{proof}
By (\cite{E2}, Corollary 1.2), Lemma \ref{lem: nonzerominor}, and Corollary \ref{cor: cordi}, the Hilbert function $h_M(t)$ of $M$ is given by 

\begin{multline}\label{eq: hilbFunOrig}
\sum_{i=1}^n \left[{t+2-a_i\choose 2}-{t+2+a_i-d\choose 2}\right]\\
+ \sum_{j=1}^{n+2} \left[{t+2 +b_j - d\choose 2} - {t+2-b_j\choose 2}\right]
\end{multline}

Since $a_1 = 0$, the maximal socle degree of $M$ is $c := d-3$ by Corollary \ref{cor: maxsocdeg}.  We first claim for $t\leq \floor$, ${t+2+a_i-d\choose 2} = 0$ and ${t+2+b_j-d\choose 2} = 0$.  By Lemma \ref{lem: codim3}, we have $a_i \leq a_n < b_n \leq b_{n+2}$, and $b_j\leq b_n$ by hypothesis.  With this, it suffices to show $\floor + 2 +b_{n+2} - d \leq 1$.  This is equivalent to showing that $b_{n+2} \leq \lfloor \frac{d}{2}\rfloor + 1$.  Hence if $d$ is even, this is equivalent to showing $2b_{n+2}\leq d+2$, and if $d$ is odd, this equivalent to showing $2b_{n+2} \leq d+1$.  These inequalities both follow immediately from the assumptions in (a) and (b), respectively.  

Using \eqref{eq: hilbFunOrig}, and the above remarks, the proof is complete.  
\end{proof}

For clarity, we note the following.  

\begin{remark}\label{rmk: hilbFun}
We use the following notation:  for $t\leq \floor$, $A_i(t) := {t+2-a_i\choose 2} = \frac{(t+2-a_i)(t+1-a_i)}{2}$ and $B_j(t) := {t+2-b_j\choose 2} = \frac{(t+2-b_j)(t+1-b_j)}{2}$, so that $h_M(t) = \sum_{i=1}^n A_i(t) - \sum_{j=1}^{n+2} B_j(t)$ by Lemma \ref{lem: hilbFunForm}.  
\end{remark}

The following is a basic fact will be used several times through the proof of Proposition \ref{prop: Munimod}, and we add it to avoid confusion.  

\begin{remark}\label{rmk: poly}
Suppose a discrete integer-valued function $f(X)$ on the closed interval $[a,b]$ has $f(t) = p(t)$, for $p(X)$ a polynomial and $t$ in $[a,b]$.  Then if $p'(t) \geq 0$ for $t\in [a,b]$, $f(X)$ is increasing, and we abuse language to say \textbf{differentiation of $f(X)$ yields an increasing function}.  
\end{remark}

\begin{proposition}\label{prop: Munimod}
Suppose $\mbk$ has characteristic not two, $d = \sum_{j=1}^{n+2} b_j -\sum_{i=1}^n a_i$, and $d' = \sum_{i=1}^n (b_i -a_i)$.  The Hilbert function of $M$ is unimodal if $a_1 = 0$ and 

\begin{enumerate}
    \item[(a)]$d$ is even and $ b_{n+2} < d' + b_{n+1} + 2$; 
    \item[(b)] $d$ is odd and $b_{n+2} < d' + b_{n+1} + 1$.  
\end{enumerate}
\end{proposition}  

\begin{proof}
From Lemma \ref{lem: hilbFunForm}, the Hilbert function $h_M(t)$ of $M$ is given by \eqref{eq: hilbFun}.  As $a_1 = 0$, the maximal socle degree of $M$ is $c := d-3$ by Corollary \ref{cor: maxsocdeg}.  Our goal in this proof is to show that $h_M$ is increasing on the interval $[0,\floor]$, so with Proposition \ref{prop: symhilb} in hand, $h_M$ is symmetric, so will achieve its absolute maximum at $\floor$ on $[0, c]$.  

Since $A_i(t) := {t+2-a_i\choose 2}$ and $B_j(t) := {t+2-b_j\choose 2}$ are nonzero only when $t\geq \max\stl a_i, b_j\str$, given $t\in [0,\floor]$, we let $j_t := \max_{t\geq b_j} j$ and $i_t := \max_{t\geq a_i} i$, and our hypotheses $a_i\leq a_{i+1}$ and $b_j\leq b_{j+1}$, along with $a_i < b_i$ from Lemma \ref{lem: codim3}, then $t$ has to be in one of the following subsets of $[0,\floor]$:  

\begin{enumerate}
    \item[(1)] $[b_{n+2}, \floor]$ if $j_t = n+2$;
    \item[(2)] $[b_{n+1}, b_{n+2}-1]$ if $j_t = n+1$; 
    \item[(3)] $[b_n, b_{n+1}-1]$ if $j_t = n$; 
    \item[(4)] for $v < n$, $[a_n, b_n-1] \cap [b_v, b_{v+1}-1]$ if $i_t = n$ and $j_t = v$
    \item[(5)] for $v \leq u < n$, $[a_u, a_{u+1}-1]\cap [b_v, b_{v+1}-1]$ if $i_t = u$ and $j_t = v$ 
    \item[(6)] for $u < n$, $[a_u, a_{u+1}-1]$ if $i_t = u$ and $j_t = -\infty$.  
\end{enumerate}

For clarity, we note we are using Lemma \ref{lem: codim3} in the following ways:  first, in the subset given in (4), we have $v < n$, as $a_n < b_n\leq b_{n+1} \leq b_{n+2}$, so that $[a_n, b_n-1]$ would have empty intersection with the appropriate interval if $v\geq n$; second, we note that we cannot $v > u$, since for $t$ in (5), this would give, using Lemma \ref{lem: codim3}, that $t\leq a_{u+1} -1 < a_{u+1} < b_{u+1} \leq b_v$

Then using \eqref{eq: hilbFun}, notice for $t$ in $(1) $, that $h_M(t) = -t^2 + ct + \alpha$, where $\alpha\in\rats$; we write $p(t) := -t^2 + ct + \alpha$.  Using this, we have the following: 

\begin{enumerate}
    \item[(2a)] if $t$ is in (2), then $h_M(t) = p(t) + B_{n+2}(t)$; 
    \item[(3a)] if $t$ is in (3), then $h_M(t) = p(t) + B_{n+2}(t) + B_{n+1}(t)$; 
    \item[(4a)] if $t$ is in (4), $h_M(t) = p(t) + \sum_{j=v+1}^{n+2} B_j(t)$;   
    \item[(5a)] if $t$ is in (5), then $h_M(t) = p(t) + \sum_{j=v+1}^{n+2} B_j(t) - \sum_{i=u+1}^n A_i(t)$; 
    \item[(6a)] if $t$ is in (6), then $h_M(t) = p(t) + \sum_{j=1}^{n+2} B_j(t) - \sum_{i=u+1}^n A_i(t)$,
\end{enumerate}

Using the equalities for $h_M(t)$ above, and the fact that the Hilbert function can be  calculated by a polynomial by at all $t\in [0,\floor]$, we can perform an analysis using the language of Remark \ref{rmk: poly} for the intervals (1)-(6).  

\begin{enumerate}
    \item[(1b)] For $t$ in (1), it is easy to see using $p(t)$ that differentiation of $h_M$ yields a strictly increasing function on (1).  
    \item[(2b)] For $t$ in (2), using (2a), differentiation of $h_M$ gives $-t + c + \frac{3-2b_{n+2}}{2}$.  Now for $t\in [b_{n+1}, b_{n+2})$, if $d$ is even, then our assumption in (a) shows that $2b_{n+2} < d+2$, hence $2b_{n+2}\leq d+1$.  As $d$ is even, we have $2b_{n+2}\leq d$.  If $d$ is odd, then our assumption in (b) gives $2b_{n+2} < d+1$, hence $2b_{n+2}\leq d$.  We have, 

    $$c+\frac{3-2b_{n+2}}{2} > \frac{d-3}{2}\geq  \Big\lfloor \frac{c}{2} \Big \rfloor\geq t. $$
    This gives that differentiation of $h_M$ on (2) yields that $h_M$ is strictly increasing on (2).  

    \item[(3b)] For $t$ in (3), using (3a), differentiation of $h_M$ gives $d'$, and we know $d' > 0$ by Lemma \ref{lem: codim3}, so that differentiation of $h_M$ yields a strictly increasing function on (3).  
    
    \item[(4b)]  For $t$ in (4), using (4a), differentiation of $h_M$ gives $(n-v)t + \sum_{j=1}^v b_j - \sum_{i=1}^na_i + \frac{3(n-v)}{2}$.  Then using Lemma \ref{lem: codim3}, and that $a_i\leq a_n$ for all $i$, we have 
    \begin{eqnarray*}
     \sum_{j=1}^v b_j -\sum_{i=1}^n a_i + \frac{3(n-v)}{2}  &\geq & v -\sum_{i=v+1}^n a_i + \frac{3(n-v)}{2} \\
    &\geq &v +  (n-v)\frac{(3-2a_n)}{2}.  
    \end{eqnarray*}
    Then for $t$ in $(4)$, we have 
    \begin{eqnarray*}
    (n-v)t + \sum_{j=1}^v b_j - \sum_{i=1}^na_i + \frac{3(n-v)}{2} &\geq & (n-v)t + v +  (n-v)\frac{(3-2a_n)}{2}\\
    &\geq & (n-v)a_n + v +  (n-v)\frac{(3-2a_n)}{2}\\
    &=& \frac{3(n-v)}{2} + v\\
   &>& 0.
    \end{eqnarray*}
    This gives that $h_M$ is a strictly increasing function on (4).  
    
    \item[(5b)] For $t$ in (5), using (5a), differentiation of $h_M$ yields $(u-v)t + \sum_{j=1}^v b_j - \sum_{i=1}^u a_i + \frac{3(u-v)}{2}$.  Then using Lemma \ref{lem: codim3}, and that $a_u\geq a_i$ for $i = v+1, v_2,\ldots, u$, we have 
    \begin{eqnarray*}
     \sum_{j=1}^v b_j - \sum_{i=1}^u a_i + \frac{3(u-v)}{2}  &\geq & v  - \sum_{i= v+1}^u a_i + \frac{3(u-v)}{2} \\
    &\geq & v - (u-v)\left(\frac{3-2a_u}{2}\right). 
    \end{eqnarray*}
    Then for $t$ in (5), we have using the above inequality, that 
    \begin{eqnarray*}
    (u-v)t + \sum_{j=1}^v b_j - \sum_{i=1}^u a_i + \frac{3(u-v)}{2} &\geq &(u-v)a_u + v - (u-v)\left(\frac{3-2a_u}{2}\right)\\
    &\geq & \frac{3(u-v)}{2} + v\\
   &>& 0 
    \end{eqnarray*}

     This gives that $h_M$ is a strictly increasing function on (5). 
     
    \item[(6b)] For $t$ in (6), using (6a), differentiation of $h_M$ gives $ut - \sum_{i=1}^u a_i + \frac{3u}{2}$.  Then for $t$ in (6), using $a_i \leq a_u$ for $i= 1,\ldots, u$, we have 

    \begin{eqnarray*}
    ut - \sum_{i=1}^u a_i + \frac{3u}{2} &\geq& a_uu - \sum_{i=1}^u a_u + \frac{3u}{2}\\
    &=& \frac{3u}{2}\\
    &>& 0.
    \end{eqnarray*}
    
    This gives that $h_M$ is a strictly increasing function on (6). 
\end{enumerate}

Then (1b) to (6b) show that the Hilbert function of $M$ is strictly increasing on $[0, \floor]$ as needed.  
\end{proof} 

\begin{remark}\label{rmk: whenMainPropHolds}
We can rewrite the inequalities in Proposition \ref{prop: Munimod} in terms of the difference of $b_{n+2}-b_{n+1}$, which is nonnegative by our assumptions.  Then by Lemma \ref{lem: codim3}, we have $n\leq d'$, so that Proposition \ref{prop: Munimod} will always hold under our blanket assumptions in the event $a_1$ is zero and $b_{n+2} - b_{n+1}\leq n$.  
\end{remark}

\section{Lefschetz Properties for $M$}\label{sec: LefProp}

We utilize the same setup in this section as in Section \ref{sec: minFreeRes}, except we suppose $\mbk$ has characteristic zero.  Set $E = \ker(\vp)$ and let $\E$ be the sheafification of $E$, so that $\E$ is a vector bundle of rank two on $\mbp^2$.  

When $n = 1$, $M = R/I$ with $I$ a complete intersection, and in \cite{CI}, conditions were sought to force the semistability of the vector bundle $\E$.  In fact, if $\ell\in R$ is general linear form and $\ol R = R/\ell R$, it was shown, using a theorem of Grauert-M{\"u}lich (\cite{VB}, pg. 206) that the first syzygy of $\ol{I}$ was given by $\ol{R}(e_1) \ds \ol{R}(e_2)$ with $|e_1-e_2| = 0$ or $1$.  This allowed for a nearly immediate conclusion that $R/I$ has the WLP.  We show that the same tools that allowed this conclusion generalize to our setting.  

Recall the graded minimal free resolution $\mbf_\bu$ of the graded $R$-module $M$ has the form:

\begin{equation}\label{eq: freeResM}
0\ra \mbf_3 \ra \mbf_2 \ra \mbf_1 \overset{\vp}{\ra} \mbf_0 \ra M\ra 0.  
\end{equation}

Where, from corollaries \ref{cor: corcj} and \ref{cor: cordi},

\begin{enumerate}
    \item $\mbf_0 = \bds_{i=1}^n R(-a_i)$;
    \item $\mbf_1 = \bds_{j=1}^{n+2} R(-b_j)$; 
    \item $\mbf_2 = \bds_{j=1}^{n+2}R(b_j-d)$;
    \item $\mbf_3 = \bds_{j=1}^{n+2}R(a_i-d)$,
\end{enumerate}

with $d = \sum_{j=1}^{n+2} b_j - \sum_{i=1}^n a_i$.  Now with $E$ and $\mce$ as above, the sheafificaiton of \eqref{eq: freeResM} gives,

\begin{equation}\label{eq: sheafResM}
0\ra \F_3 \ra \F_2\ra \E\ra 0.  
\end{equation} 
 
Here, $\F_2 =\bds_{j=1}^{n+2}\Op_{\mbp^2}(b_j-d)$ and $\F_3 = \bds_{i=1}^n \Op_{\mbp^2}(a_i-d)$.  From \eqref{eq: sheafResM}, $\E$ is a vector bundle of rank two, and the additivity of the first Chern class $\mfc_1$ gives,

\begin{eqnarray}
\mfc_1(\E) &=& \mfc_1(\mcf_2) - \mfc_1(\mcf_3)\nonumber  \\
&=&\sum_{j=1}^{n+2} (b_j-d)  - \sum_{i=1}^n (a_i-d) \nonumber \\
& =& -d.\nonumber  
\end{eqnarray}

We would like conditions that force the semistability of $\E$.  We first consider the case in which $d$ is even.  Write $d = 2e$, so that $\mfc_1(\E) = -2e$, then the normalized bundle of $\E_{\text{norm}}$ is given by $\E(e)$.  Twist \eqref{eq: sheafResM} by $e-1$ to obtain 

\begin{equation}\label{eq: evenTwist}
0\ra\F_3(e-1) \ra  \F_2(e-1) \ra \E_{\text{norm}}(-1)\ra 0.  
\end{equation} 

Assume now that $d$ is odd and choose $e$ such that $d = 2e+1$.  Then in this case, $\E_{\text{norm}} = \E(e)$ as well.  Then twist \eqref{eq: sheafResM} by $e$ to obtain,

\begin{equation}\label{eq: oddTwist}
0\ra\F_3(e) \ra  \F_2(e) \ra \E_{\text{norm}}\ra 0.  
\end{equation} 

We utilize \eqref{eq: evenTwist} and \eqref{eq: oddTwist} to give a proof of following lemma.  We note Lemma \ref{lem: semistable} is a generalization of (\cite{CI}, Lemma 2.1).  In fact, it is (\cite{CI}, Lemma 2.1) when $n = 1$ and $a_1 = 0$.  The proof is similar to (\cite{CI}, Lemma 2.1), but we provide details.  As in previous sections, we set $d' =  \sum_{i=1}^n (b_i -a_i)$.  

\begin{lemma}\label{lem: semistable}
The rank two vector bundle $\E$ on $\mbp^2$ in \eqref{eq: sheafResM} is semistable when,

\begin{enumerate}
    \item[(a)] $d$ is even and $ d' + b_{n+1}+2> b_{n+2}$;
    \item[(b)] $d$ is odd and $d'+ b_{n+1} + 1> b_{n+2}$.  
\end{enumerate}
\end{lemma} 

\begin{proof} 
Before we start, we remark that $d > 2a_n$ regardless of the parity of $d$.  To see this, note $b_j\leq b_{j+1}$ by hypothesis, and Lemma \ref{lem: codim3} gives $a_i < b_i$ for $i = 1,\ldots, n$, so that $d' > 0$ and $a_n < b_j$, for $j = n, n+1, n+2$.  This gives, 

\begin{eqnarray}\label{eqnarray: d2an}
d  &=& d' + b_{n+1}+ b_{n+2}\\
&>& d'+ 2a_n \nonumber\\
&>& 2a_n.\nonumber
\end{eqnarray}

Assume first $\mfc_1(\E)$ is even, and write $d= 2e$, so $\mfc_1(\mce) = -2e$.  Now $\E$ has rank two, so that from (\cite{VB}, Lemma 1.2.5), we see $\E$ is semistable if and only if $H^0(\mbp^2, \E_{\text{norm}}(-1)) = 0$.  Now \eqref{eq: evenTwist} is given explicitly by   

\begin{equation}\label{eq: explicitEvenRes}
0\ra \bds_{i=1}^n \Op_{\mbp^2}(a_i-e-1) \ra \bds_{j=1}^{n+2}\Op_{\mbp^2}(b_j -e-1)\ra \E_{\text{norm}}(-1)\ra 0.  
\end{equation} 

Now \eqref{eq: explicitEvenRes} is exact on global sections, so in order for semistability to hold, we need the following inequalities to hold: 

\begin{enumerate}
    \item[(i)] $b_{n+2} < e + 1$; 
    \item[(ii)] $ a_n < e +1$.
\end{enumerate}

For (i), this is equivalent to the inequality $d + 2 > 2b_{n+2}$.  Then, 

\begin{eqnarray*}
d+2-2b_{n+2} &=& d'+2 + b_{n+1} - b_{n+2}\\
&>& 0.  
\end{eqnarray*}

Here, the second inequality above holds by hypothesis.  The inequality in (ii) is equivalent to  $d+2 > 2a_n$, but we know this holds by \eqref{eqnarray: d2an}.  

For (b), write $d = 2e+1$, so $\mfc_1(\mce) = -2e-1$.  Since $\mfc_1(\mce)$ is odd and $\E$ has rank two, stability and semistability coincide by (\cite{VB}, pg. 166) and the condition for semistability is $H^0(\mbp^2, \E_{\text{norm}}) = 0$.  Now \eqref{eq: oddTwist} is given explicitly by,

\begin{equation}\label{eq: oddExplicitRes}
0\ra \bds_{i=1}^n \Op_{\mbp^2}(a_i-e-1) \ra \bds_{j=1}^{n+2}\Op_{\mbp^2}(b_j-e-1)\ra \E_{\text{norm}}\ra 0.  
\end{equation}  

Then \eqref{eq: oddExplicitRes} is exact on global sections, for the semistability of $\e$, we need the following in inequalities to hold:  

\begin{enumerate}
    \item[(iii)] $b_{n+2} < e+1$;  
    \item[(iv)] $a_n < e+1$.  
\end{enumerate}

To see (iii), note this is equivalent to $2b_{n+2} < d+1$, and ,

\begin{eqnarray*}
d+1-2b_{n+2} &=& d'+ 1 + b_{n+1} - b_{n+2}\\
&>& 0.  
\end{eqnarray*}

Here, the inequality follows from our hypothesis in (b), so this gives (iii).  Now (iv) is equivalent to showing $2a_n < d+1$, hence this holds by \eqref{eqnarray: d2an}.  
\end{proof} 

Using Lemma \ref{lem: semistable}, we can say the following about the splitting type of $\e$.

\begin{corollary}\label{cor: split}
Let $\E$ be the rank two vector bundle obtained in \eqref{eq: sheafResM} and assume that any of the conditions of Lemma \ref{lem: semistable} hold.  Then the splitting type of $\E$ is,  

\begin{equation}\label{eq: splitTypeArray}
(\la_1, \la_2) = 
\left\{\begin{array}{cc} 
(-e, -e) & d = 2e\\
(-e, -e -1) & d = 2e+1\\
\end{array}\right..
\end{equation}
\end{corollary} 

\begin{proof} 
By Lemma \ref{lem: semistable}, $\E$ is semistable.  The Grauert-M{\"u}lich (\cite{VB}, pg. 206) says that in characteristic zero the splitting type of the semistable normalized rank two vector bundle $\e_\text{norm} = \e(e)$ over $\mathbb P^2$ is, 

$$(\la_1, \la_2) = 
\left\{\begin{array}{cc} 
(0, 0) & \text{if}\hspace{.2 cm} \mfc_1(\E(e)) = 0\\
(0, -1) & \text{if}\hspace{.2 cm} \mfc_1(\E(e)) = -1\\
\end{array}\right..$$

Since $\mce$ has rank two, we have 

\begin{eqnarray*}
\mfc_1(\mce(e)) &=& \mfc_1(\E) + 2e \\
&=& -d +2e\in \stl -1, 0\str,
\end{eqnarray*}
 
which gives \eqref{eq: splitTypeArray}.  
\end{proof} 

Corollary \ref{cor: split} was crucial in \cite{CI} to showing that complete intersections have the WLP in $R$.  In fact, our generalizations of the essential lemmas of \cite{CI} show that we can generalize the main result of \cite{CI}.  

\begin{theorem}\label{thm: mainWLPthm} 
If $a_1 = 0$ and, 

\begin{enumerate}
    \item[(a)] $d$ is even and $ d' +2 + b_{n+1}> b_{n+2}$; 
    \item[(b)] $d$ is odd and $d'+1 + b_{n+1}> b_{n+2}$,
\end{enumerate}

where $d = \sum_{j=1}^{n+2} b_j - \sum_{i=1}^n a_i$, and $d' = \sum_{i=1}^n (b_i -a_i)$, then $M$ has the WLP in the sense of Definition \ref{def: wlp}.  
\end{theorem} 

The proof of Theorem \ref{thm: mainWLPthm} works entirely in the same way as the proof (\cite{CI}, Theorem 2.3), changing only what is necessary, so we omit the details.  However, we do note a couple points of caution.  As previously mentioned at the beginning of Section \ref{sec: UniAndSym}, we must understand the unimodality of the Hilbert function of $M$ before employing the mechanics of the proof of (\cite{CI}, Theorem 2.3).  This is precisely the purpose of Proposition \ref{prop: Munimod} in this context.  Moreover, it is well-known complete intersections have symmetric Hilbert functions and this is a subtle detail in the proof of (\cite{CI}, Theorem 2.3).  However, Proposition \ref{prop: symhilb} shows this the Hilbert function of $M$ is also symmetric, allowing the proof of (\cite{CI}, Theorem 2.3) to generalize to our setting.  

Theorem \ref{thm: mainWLPthm} also allows (\cite{CI}, Theorem 2.3) to be obtained as a corollary.  

\begin{corollary} 
Complete intersections in $R$ have the WLP.  
\end{corollary} 

\begin{proof} 
Suppose $f_1, f_2, f_3$ is a regular sequence with $\deg(f_j) = d_j$ and $2\leq d_1\leq d_2\leq d_3$ in $R$.  Set $I = (f_1, f_2, f_3)$.  Then it is well-known $R/I$ has a unimodal symmetric Hilbert function.  Moreover, with notation as in Theorem \ref{thm: mainWLPthm}, we have $a_1 = 0$ and $b_j = d_j$.  If $d_3 < d_1 + d_2+1$, the associated vector bundle $\e$ will be semistable by Lemma \ref{lem: semistable}, so that we can apply Theorem \ref{thm: mainWLPthm}.  Now (\cite{JW}, Corollary 3) shows that if $d_3 \geq d_1+d_2 -3$, then $R/I$ has the WLP.  
\end{proof} 

\begin{example}
Let $f_1, f_2, f_3\in R_q$  be a regular sequence with $q \geq 2$.  For $n > 1$, define $\vp:  R(-q)^{n+2} \ra R^n$ as follows:  Let $\textbf{v}$ be the row vector $[f_1, f_2, f_3,0,\ldots, 0]\in R(-q)^{n+2}$, and let $\vp$ be given by the first $n$ rows of the circulant matrix of $\textbf{v}$.  Then $\vp$ has matrix, 

$$\left[\begin{matrix} 
f_1 & f_2 & f_3 & 0 & 0 &\cdots &0 & 0  \\
0 & f_1 & f_2 & f_3 & 0 &\cdots & 0 & 0 \\
0 & 0 & f_1 & f_2 & f_3 &\cdots & 0 & 0 \\
\vdots &\vdots & \vdots & \vdots &\vdots & \vdots &\vdots &\vdots \\
0 & 0 & \cdots &\cdots& \cdots & f_1 & f_2 & f_3\\ 
\end{matrix}\right].$$ 

Let $I$ denote the ideal of $n\times n$ minors of $\vp$, then it is easy to see that $I$ has codimesnion $3$, hence $M = \ckr(\vp)$ is a graded Artinian $R$-module.  Moreover, since $\im(\vp)\subseteq \mfm$, the minimal number of generators of $M$ as an $R$-module is $n$, hence $M$ is not cyclic as $n > 1$.     

Now $d = (n+2)q$ and the conditions of Lemma \ref{lem: semistable} are satisfied regardless of the parity of $d$ since $n > 1$.  Thus $M$ has the WLP by Theorem \ref{thm: mainWLPthm}.  
\end{example}

\section{The non-Lefschetz Locus for Graded Modules}\label{sec: nonLefLocus} 

We now turn our attention to the more general setting of working over $S = \mbk[x_1,\ldots, x_r$], with $\mbk$ an algebraically closed field of characteristic zero.  All modules considered will be finitely generated.  Let $N = \bds_{j\in\ints} N_j$ be a graded Artinian module, so it has finite length.  We also use $\mfm$ to denote the graded irrelvant ideal $(x_1,\ldots, x_r)$ of $S$.  

In \cite{nll}, the authors defined what they called the \textbf{non-Lefschetz locus} for a cyclic $S$-module $S/I$.  We recall this notion and their discussion for Artinian $S$-modules.  The $S$-module structure of $N$ is determined by a sequence of $\mbk$-linear maps,

$$\phi_j: S_1\ra \hm_\mbk(N_j, N_{j+1}),$$ 

where $j$ ranges from the initial degree of $N$ to the penultimate degree where $N$ is not zero.  Since the $\mbk$-dimension of $N_j$ and $N_{j+1}$ is finite, $\phi_j(x_i)$ is a matrix of size $\dim_\mbk N_{j+1}\times \dim_\mbk N_j$; we write $\phi_j(x_i) = X_{i, j}$ in what follows.  Now given any linear form $\ell = a_1x_1 + \cdots + a_rx_r$, we have 

$$\phi_j(\ell) = a_1X_{1, j} + \cdots + a_rX_{r, j} := X_j.$$ 

We want to understand  the $X_j$ when the $a_1,\ldots, a_r$ vary, so we regard them as variables, which we call the \textbf{dual variables}.  In this setting, $X_j$ is a $\dim_\mbk N_{j+1}\times \dim_\mbk N_j$ matrix in the ring $\mbk[a_1,\ldots, a_r]$ whose entries are linear forms in the dual variables.  In particular, the scheme defined by the vanishing of the maximal minors of the matrix $X_j$, which we denote by $\mcy_j$ can viewed as lying in \textbf{dual projective space} $(\mathbb P^{r-1})^*$.  

When $\ell\in S_1$, we call $\ell$ a \textbf{Lefschetz element of} $N$ if it satisfies Definition \ref{def: wlp}.  We view the collection of Lefschetz elements as a, possibly empty, subset of $(\mbp^{r-1})^*$.  We want to know want to know what the relationship between the scheme $\mcy_j$ and the failure of $N$ to have the WLP.   

\begin{remark}
Recall that the rank of a matrix over an integral domain can be defined as the maximum $t$ such that there is a non-vanishing $t\times t$ minor.  With our definitions as above, it is easy to see the following are equivalent:  

\begin{enumerate}
    \item[(a)] $N$ does not have the WLP.   
    \item[(b)] There is a $j$ such that $X_j$ does not have maximal rank as a matrix over $\mbk[a_1,\ldots, a_r]$.  
    \item[(c)] There is a $j$ such that $\mcy_j = (\mathbb P^{r-1})^*$.  
\end{enumerate}
 
In particular, we see that $N$ has the WLP in the sense of Definition \ref{def: wlp} if and ony if there is an $\ell$ such that for all $j$, $\mcy_j \neq (\mathbb P^{r-1})^*$.  This brings us to the titular notion of this section, where we follow \cite{nll}.  
\end{remark}

\begin{definition} 
Given an Artinian $S$-module $N$, we define 

$$\ml_N := \stl [\ell] \in \mbp(S_1)\, \vert\, \text{$\ell$ is not a Lefschetz element of $N$} \str \subset (\mbp^{r-1})^*$$ 

and we call it the \textbf{non-Lefschetz locus of} $N$.  For any integer $j$, we define $\ml_{N, j}$ to be the set

$$\stl [\ell] \in \mbp(S_1)\, \vert\, \text{$\times\ell : N_j \ra N_{j+1}$ does not have maximal rank}\str \subset (\mbp^{r-1})^*.$$ 
\end{definition}   

Of course, we would like to study $\ml_{N, j}$ not just as a collection, but as a scheme.  Let $A = \mbk[a_1,\ldots ,a_r]$ denote the coordinate ring of dual projective space $(\mbp^{r-1})^*$.  We can view $\ml_{N, j}$ as the scheme defined by the maximal minors of the matrix representing the map 

$$\times\ell:  A\otimes_\mbk N_j \ra A\otimes_\mbk N_{j+1}$$ 

of free $A$-modules.  In fact, the matrix representing this map is just $X_j$.  Denote the ideal of maximal minors in $A$ defining the scheme $\ml_{N, j}$ by $I(\ml_{N, j})$.  In this way, we have $\ml_N = \bigcup_j \ml_{N, j}$ and $\ml_N$ is defined by the homogeneous ideal $I(\ml_N) = \bigcap_j I(\ml_{N, j})$.  

When studying Artinian Gorenstein \textit{algebras}, it is well-known that an algebra fails to have the WLP if injectivity fails in a single degree.  In particular, as a set, the non-Lefschetz locus is determined by a single degree (\cite{mon}, Proposition 2.1).  Moreover, it is also true that the non-Lefschetz locus is defined by a single degree \textit{scheme-theoretically} (\cite{nll}, Corollary 2.6).  While Symmetrically Gorenstein (see Definition \ref{def: symgor}) is a suitable analogue of the Gorenstein condition for Artinian modules, we cannot guarantee that certain properties of Artinian algebras with the WLP hold for all Artinian \textit{modules}.  For example, we have had to show great caution when discussing unimodality and symmetry of the Hilbert function for Symmetrically Gorenstein modules, and as such, the non-Lefschetz Locus will not be different.  

We first begin by recovering a well-known result for Artinian algebras.  The proof is roughly the same as (\cite{lp}, Proposition 3.2), but we include the details for convenience.  

\begin{proposition}\label{prop: unimodal}
Suppose $N = S^v/L$, with $L$ a homogeneous $S$-submodule of the free module $S^v$ generated by elements of positive degree with respect to the standard grading on $S^v$.  Then $N$ is a nonnegatively graded $S$-module that is generated as an $S$-module in degree zero.  If $N$ is Artinian and has the WLP, then the Hilbert function of $N$ is unimodal.  
\end{proposition} 

\begin{proof} 
Let $\mfm$ be the irrelevant ideal of $S$ and write $N = N_0\ds\cdots\ds N_c$, so that $N_c$ is nonzero and $N$ is generated by $N_0$.  Then $\mfm^iN_0$ generates $N_i$ as a vector space over $\mbk$.  Let $j\geq 0$ be the smallest integer such that $\dim_\mbk N_j > \dim_\mbk N_{j+1}$.  Since $N$ has the WLP, there is an $\ell\in S_1$ such that $\times \ell:  N_j\ra N_{j+1}$ is surjective.  Thus $\ell N_j = N_{j+1}$.  That is, $\mfm^{j+1}N_0 = \ell\mfm^jN_0$.  Hence for $i\geq j$, we have $\ell N_i = N_{i+1}$, so that $\times \ell:  N_i\ra N_{i+1}$ is surjective.  This gives,

$$v\leq \dim_\mbk N_1\leq \dim_\mbk N_2\leq\cdots\leq \dim_\mbk N_j > \dim_\mbk N_{j+1}\geq \cdots\geq \dim_\mbk N_c$$    
\end{proof} 

In more than three variable cokernels of maps as in \eqref{eq: defM} are not, in general, Symmetrically Gorenstein.  However, under mild restrictions, they fit naturally into a certain class of Artinian modules.  We follow \cite{alm} in the next definition.     

\begin{definition}\label{def: level} 
We define the \textbf{socle} of $N$, denoted by $\text{Soc}(N)$, as 

\begin{equation*}
(0:_N \mfm) := \stl x\in N \, \vert\, \text{$ax = 0$ for all $a\in \mfm$}\str.
\end{equation*}

The \textbf{socle degree} of $N$ 
We say that an Artinian $S$-module $N$ is \textbf{level} if it is generated by $N_0$ as an $S$-module and $\text{Soc}(N)  = N_c$ for some $c$.  In particular, $N_i = 0$ for $i > c$ if $N$ is level.  
\end{definition}  

Recall from Definition \ref{def: dual} that if $N$ is an $S$-module, the $\mbk$-dual of $N$ is the graded $S$-module $N^\vee :=\hm_\mbk(N, \mbk)$ with grading such that $N^\vee_j := \hm_\mbk(N_{-j}, \mbk)$.  In particular, if $N$ is nonnegatively graded Artinian $S$-module, say $N = N_0\ds\cdots \ds N_c$ with $N_c$ nonzero, then $N^\vee(-c) $ is Artinian and nonnegatively graded with maximal socle degree $c$.  Even more is true.  

\begin{proposition}(\cite{alm}, Proposition 2.3)\label{prop: boij}
Assume that $N$ is a nonnegatively graded Artinian $S$-module that is level in the sense of Definition \ref{def: level}.  If $\text{Soc}(N)  = N_c$, then $N^\vee(-c)$ is an Artinian level $S$-module.  
\end{proposition} 

We utilize Proposition \ref{prop: boij} to recover a well-known result for level algebras (\cite{mon}, Proposition 2.1).  

\begin{proposition}\label{prop: graded}
Suppose $N = S^v/L$ with $L$ a homogeneous $S$-submodule generated by elements of positive degree with respect to the standard grading on $S^v$.  Suppose $N$ is Artinian with $N = N_0\ds\cdots\ds N_c$.  Let $\ell$ be a linear form in $S$.  Denote by $\Psi_t:  N_t\ra N_{t+1}$ for $t\geq 0$  multiplication by $\ell$ on $N_t$.    

 \begin{enumerate}
     \item[(a)] If $\Psi_{t_0}$ is surjective for some $t_0$, then $\Psi_t$ is surjective for all $t\geq t_0$.
     \item[(b)]  Suppose $N$ is level in the sense of Definition \ref{def: level}.  If $\Psi_{t_0}$ is injective for some $t_0\geq 0$  then $\Psi_t$ is injective for all $t\leq t_0$. 
     \item[(c)]  In particular, if $N$ is level and there is a $t_0$ such that $\dim_\mbk N_{t_0} = \dim_\mbk N_{t_0+1}$, then $N$ has the WLP if and only if $\Psi_{t_0}$ is injective. 
 \end{enumerate}
\end{proposition} 

\begin{proof} 
(a) was shown in the proof of Proposition \ref{prop: unimodal}, and (c) follows immediately from (a) and (b), so we show (b).  

By hypothesis, $\text{Soc}(N) = (0:_N \mfm) = N_c$, so that $N^\vee(-c)$ is level by Proposition \ref{prop: boij}, so is generated in degree $0$.  Now we can consider multiplication by $\ell$ on $N^\vee(-c)$.  Write $t_0  = c-s_0$, for some $s_0$ between $0$ and $c$.  Then the injectivity of $\Psi_{t_0}$ gives that $\times\ell:  N^\vee(-c)_{s_0-1}\ra N^\vee(-c)_{s_0}$ is surjective.  From (a), we obtain $\times \ell:  N^\vee(-c)_s\ra N^\vee(-c)_{s+1}$ is surjective for $s\geq s_0-1$.  Dualizing, we see $\times \ell:  \hm_\mbk(N^\vee(-c)_{s+1}, \mbk)\ra \hm_\mbk(N^\vee(-c)_s, \mbk)$ is injective, so that $\Psi_{c-s-1}$ is injective.  Since every $t\leq t_0$ has the form $c-s - 1$ for some $s\geq s_0 - 1$, we are done.   
\end{proof} 

Now the next proposition is crucial to our endeavors and it is an analogue of (\cite{nll}, Proposition 2.5).  The proof of (\cite{nll}, Proposition 2.5) works by only changing what is necessary, so we omit the details.    

\begin{proposition}\label{prop: crux}
Suppose that $N$ is an Artinian nonnegatively graded $S$-module with Hilbert function $h_N$.  If $h_N(i) \leq h_N(i+1) \leq h_N(i+2)$ and $\soc(N)_i = 0$, then $I(\ml_{N, i+1}) \subseteq I(\ml_{N, i})$.  
\end{proposition} 

With Proposition \ref{prop: crux} in hand, we have the following.   

\begin{corollary}\label{cor: levelLocus}
Suppose $N$ is a nonnegatively graded Artinian level $S$-module of maximal socle degree $c$.  There is a $j$ such that 
$$\ml_N = \ml_{j-1, N} \cup \ml_{j, N}$$ 
\end{corollary}

\begin{proof} 
Suppose $N$ does not have the WLP.  Let $\ell\in S_1$ be a general linear form such that there is a $j$ so that $\times\ell:  N_j \ra N_{j+1}$ does not have maximal rank.  Then $I(\mcl_{N, j}) = 0$, so $\mcl_{N, j} = (\mbp^{r-1})^*$, so that $\ml_N = \ml_{N, j}  = (\mbp^{r-1})^*$.

Now suppose $N$ has the WLP.  Then its Hilbert function is unimodal by Proposition \ref{prop: unimodal}, so that there is a $j$ such that $h_N(i)\leq h_N(i+1)$ for $i <  j$ and $h_N(i)\geq h_N(i+1)$ for $j \leq i$.  For $i < j$, we apply Proposition \ref{prop: crux} to see

$$I(\ml_{N, j-1})\subseteq I(\ml_{N, j-2})\subseteq \cdots\subseteq I(\ml_{N, 1})\subseteq I(\ml_{N, 0}).$$ 

For $i=0,\ldots, j-1$, this gives

\begin{equation}\label{eq: levelSub}
\ml_{N, i}\subseteq \ml_{N, j-1}.  
\end{equation}

 Now $N^\vee(-c)$ is also an Artinian level module with socle in degree $c$ by Proposition \ref{prop: boij}.  Since $N^\vee(-c)_i = \hm_\mbk(N_{c-i}, \mbk)$, for $i = 0,\ldots, c-j-1$,
\begin{equation}\label{eq: levelHilbFun}
h_{N^\vee(-c)}(i)\leq h_{N^\vee(-c)}(i+1).  
\end{equation}

Then for $i=0,\ldots ,c-j-1$, using \eqref{eq: levelHilbFun} and Proposition \ref{prop: crux}, we obtain 

$$I(\ml_{ N^\vee(-c), c-j-1})\subseteq I(\ml_{N^\vee(-c), c-j-2})\subseteq \cdots \subseteq I(\ml_{N^\vee(-c), 0}).$$  

That is, for $i = 0,\ldots, c-j-1$, we have 

\begin{equation}\label{eq: levelDualContain}
\ml_{N^\vee(-c),i }\subseteq \ml_{N^\vee(-c), c-j-1}.  
\end{equation}

Now $I(\ml_{N, i})$ is defined as an ideal by the vanishing of the maximal minors of the $\mbk$-linear map $\phi_i:  S_1\ra \hm_\mbk(N_i, N_{i+1})$, and the corresponding map for $ S_1\ra \hm_\mbk(N^\vee(-c)_i, N^\vee(-c)_{i+1})$ is given by $ \phi_{c-i-1}^T$, the transpose of $\phi_{c-i-1}$.  This gives, 

\begin{equation}\label{eq: idealTrans}
I(\ml_{N^\vee(-c), i}) = I(\ml_{N, c-i-1}).  
\end{equation}

Then using \eqref{eq: levelDualContain} and \eqref{eq: idealTrans} , we have for $i=0,\ldots, c-j-1$, 

\begin{equation}\label{eq: levelSup}
\ml_{c-i-1, N}\subseteq \ml_{j, N}.  
\end{equation}

Combining \eqref{eq: levelSub} and \eqref{eq: levelSup},  gives the statement when $N$ has the WLP.   
\end{proof} 

Now Corollary \ref{cor: levelLocus} provides us with a nice decomposition of $\ml_N$ in the case that $N$ is Artinian and level, but pinpointing the $j$ for which this occurs can often be difficult.  However, we can pinpoint the $j$ for which this occurs in the case $N$ is Symmetrically Gorenstein, provided we know that Symmetrically Gorenstein modules are also level.  It is well-known Artinian Gorenstein algebras are level, and we answer this affirmatively for Symmetrically Gorenstein modules below.  

\begin{lemma}\label{lem: symgorlevel}
Suppose $N = S^v/L$, where $L$ is a homogeneous submodule of $S^v$ generated by elements of positive degree with respect to the standard grading on $S^v$.  If $N$ is Symmetrically Gorenstein, then $N$ is level.   
\end{lemma}  

\begin{proof} 
If $\mbg_\bu$ is the minimal free resolution of $N$, we have $\mbg_0 = S^v$.  As $N$ is Symmetrically Gorenstein, by Theorem \ref{thm: MKthm}, the last free module in $\mbg_\bu$ is $(\mbg_0)^{\vee d} = S(-d)^v$, where $d = c+r$ and $c$ is the maximal socle degree of $N$.  By Lemma \ref{lem: socle}, $N$ is level.    
\end{proof}

\begin{remark}\label{rmk: maxvalue}
If the Hilbert function $h_N$ of the Artinian module $ N= N_0\ds\cdots\ds N_c$ is symmetric and unimodal, then $h_N$ achieves its maximum value at $\lfloor\frac{c}{2}\rfloor$. 
\end{remark} 

We can now generalize (\cite{nll}, Corollary  2.7).  

\begin{proposition}\label{prop: symgorlocus}
Suppose $N = N_0 \ds\cdots \ds N_c$ is Symmetrically Gorenstein $S$-module with $N_0\neq 0$ and $N_c\neq 0$.  Then $\ml_N = \ml_{N, j}$, where $j = \lfloor \frac{c-1}{2}\rfloor$.  
\end{proposition}  

\begin{proof}
The Hilbert function of $N$ is symmetric $N$ by Lemma \ref{lem: hilbfunction}.  

Suppose first $N$ does not have the WLP.  Then the symmetry of the Hilbert function and Proposition \ref{prop: graded} say that for any general linear form $\ell$, $\times \ell$ cannot induce a map of maximal rank from $N_j\ra N_{j+1}$.  In this case, we have $I(\ml_{N, j}) = 0$, giving $\ml_N = \ml_{N,j} = (\mbp^{r-1})^*$.

Suppose $N$ has the WLP.  Then the Hilbert function of $N$ is unimodal by Proposition \ref{prop: unimodal}.  As the Hilbert function of $N$ is symmetric, by Remark \ref{rmk: maxvalue} the Hilbert function of $N$ assumes its maximum value at $\floor$.  By Lemma \ref{lem: symgorlevel}, $N$ is level, so that by the proof of Corollary \ref{cor: levelLocus}, we have 

\begin{equation}\label{eq: NLLsymgor}
 \ml_N = \ml_{\floor-1, N} \cup \ml_{\floor, N }. 
\end{equation}

If $c$ is odd, write $ c = 2k+1$, so that $ j = k$, and the symmetry of the Hilbert function gives $h_N(k+1) = h_N(k)$.  Then by Proposition \ref{prop: crux}, $I(\ml_{k, N})\subseteq I(\ml_{k-1, N})$, hence $\ml_N = \ml_{j, N}$.  

If $c$ is even, write $c  = 2k$, so that $j = k-1$.  Since $h_N(k-1) = h_N(k+1)$, and $N$ is Symmetrically Gorenstein, so self-dual by Lemma \ref{lem: hilbfunction}, so the map $\phi_{k-1}: S_1\ra \hm_\mbk(N_{k-1}, N_k)$ is the dual of the corresponding map $\phi_k$, hence $I(\ml_{k-1, N}) = I(\ml_{k, N})$.  This gives $\ml_N = \ml_{N, j}$.  
\end{proof}

\begin{corollary}
Suppose $R = \mbk[x, y, z]$.  Let $\vp$ be a degree zero graded map from $\bds_{j=1}^{n+2} R(-b_j)$ to $R^n$ $(n > 0$), where $\vp = (\vp_{ij})$ and $\vp_{ij}$ is either zero or of positive degree and $b_1 \leq \cdots \leq b_{n+2}$.  Suppose the ideal of maximal minors of $\vp$ has codimension three, so that the cokernel of $\vp$, denoted by $M$, is Artinian.  Then $\ml_M = \ml_{M, \lfloor \frac{d-4}{2}\rfloor}$, where $d = \sum b_j$.  
\end{corollary}  

\begin{proof} 
By Corollary \ref{cor: maxsocdeg}, $M$ has maximal socle degree $d-3$.  By Proposition \ref{prop: symhilb}, $M$ is nonnegatively graded and Symmetrically Gorenstein, hence we may apply Proposition  \ref{prop: symgorlocus} to obtain the result.    
\end{proof}

\section*{Acknowledgements}

The author would like to thank Gioia Failla and Chris Peterson for their helpful comments in the preparation of this manuscript.  




\bibliographystyle{elsarticle-num-names} 
\bibliography{symUniLef.bib}

\begin{thebibliography}{16}
\expandafter\ifx\csname natexlab\endcsname\relax\def\natexlab#1{#1}\fi
\providecommand{\url}[1]{\texttt{#1}}
\providecommand{\href}[2]{#2}
\providecommand{\path}[1]{#1}
\providecommand{\DOIprefix}{doi:}
\providecommand{\ArXivprefix}{arXiv:}
\providecommand{\URLprefix}{URL: }
\providecommand{\Pubmedprefix}{pmid:}
\providecommand{\doi}[1]{\href{http://dx.doi.org/#1}{\path{#1}}}
\providecommand{\Pubmed}[1]{\href{pmid:#1}{\path{#1}}}
\providecommand{\bibinfo}[2]{#2}
\ifx\xfnm\relax \def\xfnm[#1]{\unskip,\space#1}\fi
\bibitem[{Migliore and Nagel(2013)}]{MN}
\bibinfo{author}{J.~Migliore}, \bibinfo{author}{U.~Nagel},
\newblock \bibinfo{title}{Survey article: A tour of the weak and strong lefschetz properties},
\newblock \bibinfo{journal}{Journal of Commutative Algebra} \bibinfo{volume}{5} (\bibinfo{year}{2013}) \bibinfo{pages}{329--358}.
\bibitem[{Favacchio and Thieu(2012)}]{FD}
\bibinfo{author}{G.~Favacchio}, \bibinfo{author}{P.~D. Thieu},
\newblock \bibinfo{title}{On the weak lefschetz propery for graded modules over $k[x, y]$},
\newblock \bibinfo{journal}{Le Matematiche} \bibinfo{volume}{67} (\bibinfo{year}{2012}) \bibinfo{pages}{223--235}.
\bibitem[{Harima et~al.(2003)Harima, Migliore, Nagel, and Watanabe}]{CI}
\bibinfo{author}{T.~Harima}, \bibinfo{author}{J.~C. Migliore}, \bibinfo{author}{U.~Nagel}, \bibinfo{author}{J.~Watanabe},
\newblock \bibinfo{title}{The weak and strong lefschetz properties for artinian $k$-algebras},
\newblock \bibinfo{journal}{Journal of Algebra} \bibinfo{volume}{262} (\bibinfo{year}{2003}) \bibinfo{pages}{99--126}.
\bibitem[{Failla et~al.(2021)Failla, Flores, and Peterson}]{ffp}
\bibinfo{author}{G.~Failla}, \bibinfo{author}{Z.~Flores}, \bibinfo{author}{C.~Peterson},
\newblock \bibinfo{title}{On the weak lefschetz property for vector bundles on $\mbp^2$},
\newblock \bibinfo{journal}{Journal of Algebra} \bibinfo{volume}{568} (\bibinfo{year}{2021}) \bibinfo{pages}{22--34}.
\bibitem[{Kunte(2011)}]{MK}
\bibinfo{author}{M.~Kunte},
\newblock \bibinfo{title}{Gorenstein modules of finite length},
\newblock \bibinfo{journal}{Mathematische Nachrichten} \bibinfo{volume}{284} (\bibinfo{year}{2011}) \bibinfo{pages}{899--919}.
\bibitem[{Boij et~al.(2018)Boij, Migliore, Miró-Roig, and Nagel}]{nll}
\bibinfo{author}{M.~Boij}, \bibinfo{author}{J.~Migliore}, \bibinfo{author}{R.~M. Miró-Roig}, \bibinfo{author}{U.~Nagel},
\newblock \bibinfo{title}{The non-lefschetz locus},
\newblock \bibinfo{journal}{Journal of Algebra} \bibinfo{volume}{505} (\bibinfo{year}{2018}) \bibinfo{pages}{288--320}.
\bibitem[{Boij(2000)}]{alm}
\bibinfo{author}{M.~Boij},
\newblock \bibinfo{title}{Artin level modules},
\newblock \bibinfo{journal}{Journal of Algebra} \bibinfo{volume}{226} (\bibinfo{year}{2000}) \bibinfo{pages}{361--374}.
\bibitem[{Marangone(2023)}]{EM}
\bibinfo{author}{E.~Marangone},
\newblock \bibinfo{title}{The non-lefschetz locus of vector bundles of rank $2$ over $\mbp^2$},
\newblock \bibinfo{journal}{Journal of Algebra} \bibinfo{volume}{630} (\bibinfo{year}{2023}) \bibinfo{pages}{297--316}.
\bibitem[{Eisenbud(1995)}]{E1}
\bibinfo{author}{D.~Eisenbud}, \bibinfo{title}{Commutative Algebra with a View Toward Algebraic Geometry}, \bibinfo{publisher}{Springer-Verlag New York}, \bibinfo{year}{1995}.
\bibitem[{Giusti and Merle(1982)}]{sparse}
\bibinfo{author}{M.~Giusti}, \bibinfo{author}{M.~Merle},
\newblock \bibinfo{title}{Singularit\'{e}s isol\'{e}es et sections planes de vari\'{e}t\'{e}s d\'{e}terminantielles},
\newblock in: \bibinfo{editor}{J.~M. Aroca}, \bibinfo{editor}{R.~Buchweitz}, \bibinfo{editor}{M.~Giusti}, \bibinfo{editor}{M.~Merle} (Eds.), \bibinfo{booktitle}{Algebraic Geometry}, \bibinfo{publisher}{Springer Berlin Heidelberg}, \bibinfo{year}{1982}, pp. \bibinfo{pages}{103--118}.
\bibitem[{Kustin and Ulrich(1992)}]{soc}
\bibinfo{author}{A.~R. Kustin}, \bibinfo{author}{B.~Ulrich},
\newblock \bibinfo{title}{If the socle fits},
\newblock \bibinfo{journal}{Journal of Algebra} \bibinfo{volume}{147} (\bibinfo{year}{1992}) \bibinfo{pages}{63--80}.
\bibitem[{Eisenbud(2005)}]{E2}
\bibinfo{author}{D.~Eisenbud}, \bibinfo{title}{The Geometry of Syzygies}, \bibinfo{publisher}{Springer Science+Business Media}, \bibinfo{year}{2005}.
\bibitem[{Okonek et~al.(1980)Okonek, Schneider, and Spindler}]{VB}
\bibinfo{author}{C.~Okonek}, \bibinfo{author}{M.~Schneider}, \bibinfo{author}{H.~Spindler}, \bibinfo{title}{Vector Bundles on Complex Projective Space}, \bibinfo{publisher}{Birkh\"{a}user Boston}, \bibinfo{year}{1980}.
\bibitem[{Watanabe(1998)}]{JW}
\bibinfo{author}{J.~Watanabe},
\newblock \bibinfo{title}{A note on complete intersections of height three},
\newblock \bibinfo{journal}{Proceedings of the American Mathematical Society} \bibinfo{volume}{126} (\bibinfo{year}{1998}) \bibinfo{pages}{3161--3168}.
\bibitem[{Migliore et~al.(2011)Migliore, Miro\`{o}-Roig, and Nagel}]{mon}
\bibinfo{author}{J.~C. Migliore}, \bibinfo{author}{R.~M. Miro\`{o}-Roig}, \bibinfo{author}{U.~Nagel},
\newblock \bibinfo{title}{Monomial ideals, almost complete intersections and the weak lefschetz property},
\newblock \bibinfo{journal}{Transactions of the American Mathematical Society} \bibinfo{volume}{363} (\bibinfo{year}{2011}) \bibinfo{pages}{229--257}.
\bibitem[{Harima et~al.(2013)Harima, Maeno, Morita, Yasuhide~Numata, and Watanabe}]{lp}
\bibinfo{author}{T.~Harima}, \bibinfo{author}{T.~Maeno}, \bibinfo{author}{H.~Morita}, \bibinfo{author}{A.~W. Yasuhide~Numata}, \bibinfo{author}{J.~Watanabe}, \bibinfo{title}{The Lefschetz Properties}, \bibinfo{publisher}{Springer-Verlag Berlin Heidelberg}, \bibinfo{year}{2013}.

\end{thebibliography}


\end{document}